\documentclass[11pt,letterpaper]{amsart}
\usepackage[english]{babel}
\usepackage[colorlinks]{hyperref}
\usepackage[dvipsnames]{xcolor}
\usepackage{xstring}
\usepackage{seqsplit}

\usepackage[shortlabels]{enumitem}
\setenumerate{label=(\roman*)}

\newcommand\myshade{100}
\hypersetup{
  linkcolor  = RoyalBlue!\myshade!black,
  citecolor  = ForestGreen!\myshade!black,
  urlcolor   = RedOrange!\myshade!black,
  colorlinks = true,
}

\DeclareRobustCommand{\faGithub}{{\usefont{U}{fontawesomebrands0}{regular}{n}\char167}}
\DeclareRobustCommand{\faFolderOpen}{{\usefont{U}{fontawesomefree1}{solid}{n}\char105}}
\DeclareRobustCommand{\faFileCode}{{\usefont{U}{fontawesomefree1}{solid}{n}\char68}}
\colorlet{githubrefcolor}{RedOrange!\myshade!black}
\makeatletter
\newcommand{\github@breakable@text}[1]{%
  \begingroup
  \edef\github@breakable@result{#1}%
  \texttt{\expandafter\seqsplit\expandafter{\github@breakable@result}}%
  \endgroup
}
\newcommand{\github@breakable@literal}[1]{%
  \begingroup
  \edef\github@text{\detokenize{#1}}%
  \github@breakable@text{\github@text}%
  \endgroup
}
\newcommand{\github@ref}[3]{%
  \href{#3}{\textcolor{githubrefcolor}{#1\,#2}}%
}
\newcommand{\github@ref@literal}[3]{%
  \href{#3}{\textcolor{githubrefcolor}{#1\,\github@breakable@literal{#2}}}%
}
\newcommand{\github@faicon}[1]{{\scriptsize #1}}
\DeclareRobustCommand{\leanicon}{%
  \begingroup
  \scriptsize
  \leavevmode
  \ooalign{%
    \hfil$\bm\forall$\hfil\cr
    \hfil\kern0.018em$\bm\forall$\hfil\cr
    \hfil\kern0.036em$\bm\forall$\hfil\cr
    \hfil\kern-0.018em$\bm\forall$\hfil\cr
    \hfil\kern-0.036em$\bm\forall$\hfil\cr
  }%
  \endgroup
}
\newcommand{\github@strip@suffix}[1]{%
  \IfSubStr{#1}{?}{\StrBefore{#1}{?}[#1]}{}%
  \IfSubStr{#1}{\#}{\StrBefore{#1}{\#}[#1]}{}%
}
\newcommand{\github@repo@text}[1]{%
  \begingroup
  \StrBehind{#1}{github.com/}[\github@path]%
  \github@strip@suffix{\github@path}%
  \IfSubStr{\github@path}{/}{%
    \StrBehind{\github@path}{/}[\github@afterowner]%
    \IfSubStr{\github@afterowner}{/}%
      {\StrBefore{\github@afterowner}{/}[\github@text]}%
      {\let\github@text\github@afterowner}%
  }{%
    \let\github@text\github@path
  }%
  \github@breakable@text{\github@text}%
  \endgroup
}
\newcommand{\github@path@text}[2]{%
  \begingroup
  \IfSubStr{#1}{/#2/}{%
    \StrBehind{#1}{/#2/}[\github@path]%
    \github@strip@suffix{\github@path}%
    \IfSubStr{\github@path}{/}%
      {\StrBehind{\github@path}{/}[\github@text]\github@breakable@text{\github@text}}%
      {\github@repo@text{#1}}%
  }{%
    \def\github@text{#1}%
    \github@breakable@text{\github@text}%
  }%
  \endgroup
}
\newcommand{\github@line@text}[1]{%
  \begingroup
  \StrBehind{#1}{/blob/}[\github@path]%
  \github@strip@suffix{\github@path}%
  \IfSubStr{\github@path}{/}%
    {\StrBehind{\github@path}{/}[\github@file]}%
    {\StrBehind{#1}{github.com/}[\github@file]\github@strip@suffix{\github@file}}%
  \StrBehind{#1}{\#}[\github@line]%
  \IfSubStr{\github@line}{?}{\StrBefore{\github@line}{?}[\github@line]}{}%
  \github@breakable@text{\github@file}\space\github@breakable@text{\github@line}%
  \endgroup
}
\newcommand{\github@repo@default}[1]{%
  \github@ref{\github@faicon{\faGithub}}{\github@repo@text{#1}}{#1}%
}
\newcommand{\github@folder@default}[1]{%
  \github@ref{\github@faicon{\faFolderOpen}}{\github@path@text{#1}{tree}}{#1}%
}
\newcommand{\github@file@default}[1]{%
  \github@ref{\github@faicon{\faFileCode}}{\github@path@text{#1}{blob}}{#1}%
}
\newcommand{\github@line@default}[1]{%
  \github@ref{\leanicon}{\github@line@text{#1}}{#1}%
}
\newcommand{\github@tree@default}[1]{%
  \begingroup
  \StrBehind{#1}{/tree/}[\github@treepath]%
  \github@strip@suffix{\github@treepath}%
  \IfSubStr{\github@treepath}{/}%
    {\endgroup\github@folder@default{#1}}%
    {\endgroup\github@repo@default{#1}}%
}
\newcommand{\github@auto@default}[1]{%
  \IfSubStr{#1}{/blob/}{%
    \IfSubStr{#1}{\#L}%
      {\github@line@default{#1}}%
      {\github@file@default{#1}}%
  }{%
    \IfSubStr{#1}{/tree/}%
      {\github@tree@default{#1}}%
      {\github@repo@default{#1}}%
  }%
}
\newcommand{\github@auto@literal}[2]{%
  \IfSubStr{#2}{/blob/}{%
    \IfSubStr{#2}{\#L}%
      {\github@ref@literal{\leanicon}{#1}{#2}}%
      {\github@ref@literal{\github@faicon{\faFileCode}}{#1}{#2}}%
  }{%
    \IfSubStr{#2}{/tree/}%
      {\github@ref@literal{\github@faicon{\faFolderOpen}}{#1}{#2}}%
      {\github@ref@literal{\github@faicon{\faGithub}}{#1}{#2}}%
  }%
}
\NewDocumentCommand{\gh}{om}{%
  \IfNoValueTF{#1}%
    {\github@auto@default{#2}}%
    {\github@auto@literal{#1}{#2}}%
}
\makeatother

\usepackage{amsmath}
\usepackage{amssymb}
\usepackage{amsfonts}
\usepackage{amsthm}
\usepackage{bm}
\usepackage{bbold}

\usepackage[capitalise]{cleveref}
\crefformat{equation}{(#2#1#3)}
\crefname{subsection}{Section}{Sections}

\theoremstyle{plain}
\newtheorem{thm}{Theorem}[section]
\newtheorem{lem}[thm]{Lemma}
\newtheorem{prop}[thm]{Proposition}
\newtheorem{cor}[thm]{Corollary}

\theoremstyle{definition}
\newtheorem{defn}[thm]{Definition}
\newtheorem{example}[thm]{Example}

\theoremstyle{remark}
\newtheorem*{rem}{Remark}

\ExplSyntaxOn
\NewDocumentCommand{\DeclareLeanTheoremEnvironment}{m}
  {
    \cs_new_eq:cc { lean_old_#1: } { #1 }
    \cs_new_eq:cc { lean_old_end_#1: } { end#1 }
    \RenewDocumentEnvironment{#1}{ o d<> }
      {
        \IfNoValueTF{##1}
          { \use:c { lean_old_#1: } }
          { \use:c { lean_old_#1: } [##1] }
        \IfNoValueF{##2}
          { \nobreak\textup{##2}\space }
      }
      { \use:c { lean_old_end_#1: } }
  }
\ExplSyntaxOff

\DeclareLeanTheoremEnvironment{thm}
\DeclareLeanTheoremEnvironment{lem}
\DeclareLeanTheoremEnvironment{prop}
\DeclareLeanTheoremEnvironment{cor}
\DeclareLeanTheoremEnvironment{defn}

\makeatletter
\def\@fnsymbol#1{\ensuremath{\ifcase#1\or \dagger\or \ddagger\or
           \dagger\dagger
           \or \ddagger\ddagger \else\@ctrerr\fi}}
\makeatother

\newcommand\N{\ensuremath{\mathbb{N}}}
\newcommand\R{\ensuremath{\mathbb{R}}}
\renewcommand\O{\ensuremath{\varnothing}}

\newcommand{\el}[1]{\ensuremath{\ell_{#1}}}
\newcommand{\Lr}[1]{\ensuremath{\mathcal{L}_r(#1)}}

\DeclareMathOperator{\supp}{supp}
\DeclareMathOperator{\barespn}{span}

\usepackage[backend=biber,maxnames=999,giveninits=true,style=numeric,sortcites=true,datamodel=mrnumber,isbn=false,url=false,doi=false]{biblatex}
\usepackage{csquotes}

\DeclareFieldFormat{title}{\myhref{\mkbibemph{#1}}}
\DeclareFieldFormat
  [article,inbook,incollection,inproceedings,patent,thesis,unpublished]
  {title}{\myhref{\mkbibquote{#1\isdot}}}

\newcommand{\myhref}[1]{%
 \ifboolexpr{%
   test {\ifhyperref}
   and
   not test {\iftoggle{bbx:url}}
   and
   not test {\iftoggle{bbx:doi}}
  }
  {\iffieldundef{doi}
     {\iffieldundef{url}
        {#1}
        {\href{\strfield{url}}{#1}}}
     {\href{https://doi.org/\strfield{doi}}{#1}}}
  {#1}%
}

\DeclareFieldFormat{mrnumber}{%
  MR\addcolon\space
  \ifhyperref
    {\href{http://www.ams.org/mathscinet-getitem?mr=#1}{\nolinkurl{#1}}}
    {\nolinkurl{#1}}}

\renewbibmacro*{doi+eprint+url}{%
  \printfield{mrnumber}%
}

\DeclareFieldFormat{eid}{\printtext{article no.} #1}

\newcommand{\one}{\mathbb{1}}
\DeclareMathOperator{\id}{id}

\title[Wickstead's conjecture on positive projections]{Wickstead's
conjecture on positive projections and non-representable Banach
lattice algebras}
\author{David Muñoz-Lahoz}
\address{Instituto de Ciencias Matemáticas, Universidad Autónoma de
Madrid, C/ Nicolás Cabrera 13–15, 28049 Madrid, Spain}
\email{david.munnozl (at) uam (dot) es}
\thanks{Research supported by an FPI–UAM 2023 contract (funded by
Universidad Autónoma de Madrid) and by grants PID2024-162214NB-I00 and
CEX2023-001347-S (funded by MCIN/AEI/10.13039/501100011033).}
\date{August 10, 2026}

\begin{document}

\begin{abstract}
    Let $X$ be a Dedekind complete Banach lattice, and let $P\colon
    X\to X$ be a positive projection for which the largest central
    operator below $P$ is $\alpha \id_X$, for some $\alpha \ge 0$.
    Wickstead conjectured that $\alpha $ must either be $0$ or $1/n$,
    for some $n \in \N$, and proved it for finite-dimensional $X$.
    In this paper, we show that the conjecture holds in general. As a
    consequence, we settle the representation problem for Banach
    lattice algebras: we show that there exist Banach lattice
    algebras of dimension $2$ that do not admit a faithful
    representation as regular operators on any Dedekind complete
    Banach lattice. All the main results in this paper have been
    formalized in Lean 4 using the Banach lattice Lean library.
\end{abstract}

\maketitle

\section{Introduction}

In \cite{wickstead2017_two}, A.\ W.\ Wickstead formulated the following question. Let $X$ be
a Dedekind complete Banach lattice, and let $P\colon X\to X$ be a positive
projection for which there exists an $\alpha \ge 0$ satisfying:
\begin{enumerate}
\item $\alpha \id_X\le P$, and
\item whenever $T$ is a positive operator such that both $T\le
    \lambda \id_X$, for some $\lambda \ge 0$, and $T\le P$, then
    also $T\le \alpha \id_X$.
\end{enumerate}
Which values can $\alpha $ take? Wickstead conjectured that either
$\alpha =0$ or $\alpha =1/n$, for some $n \in \N$. The first main goal
of this paper is to prove that this is indeed the case
(\cref{cor:wickstead_question}).

When $X$ is finite-dimensional, conditions (i) and (ii)
imply that the diagonal elements in the matrix representing
$P$ are all equal to $\alpha $. Thus we will refer to these
two conditions by saying that $P$ has constant diagonal $\alpha
$. When $X=C(K)$, for some compact Hausdorff space $K$, or $X=L_1(\mu
)$, for some finite measure $\mu $, that $P$ has constant diagonal
$\alpha $ means that the biggest multiplication operator (by an
element of $C(K)$ or of $L_\infty (\mu )$, respectively) below $P$ is
$\alpha \id_X$.

The motivation behind this question lies in the representation of
Banach lattice algebras. A Banach lattice algebra is a Banach lattice
together with a Banach algebra structure for which the product of
positive elements is positive. One of the most prominent questions
regarding Banach lattice algebras is whether every such object can be
faithfully represented (both as a lattice and as an algebra) in $\Lr
X$, the space of regular operators on some Dedekind complete Banach
lattice $X$ \cite{wickstead2017_questions}. This pursuit of concrete
realization is heavily inspired by foundational representation
theorems in operator algebras, where abstractly defined structures are
successfully shown to be isomorphic to concrete algebras of operators.
A prime example is the Gelfand--Naimark--Segal theorem,
which establishes that any abstract $C^*$-algebra can be faithfully
represented as a closed subalgebra of bounded operators on a Hilbert
space. A positive answer to the representation problem for Banach
lattice algebras would provide a similar structural cornerstone,
allowing abstract lattice algebras to be analyzed through the
concrete, well-understood properties of regular operators on Dedekind
complete Banach lattices.

The second main goal of this paper, and the one that motivates the
first, is to settle this representation problem in the negative. We
exhibit unital Banach lattice algebras of dimension $2$ that admit
no faithful representation---unital or not---on any Dedekind
complete vector lattice (\cref{ex:final}). This answers \cite[Question
5.3]{wickstead2017_questions} and \cite[Problem
2]{wickstead2017_open} in the negative.

Wickstead studied in \cite{wickstead2017_two} the representability of
all unital two-dimensional Banach lattice algebras. He noted that, in
$\R^2$ with the pointwise product, there exists a family of lattice
cones for which the resulting Banach lattice algebras admit a unital
representation if and only if $\alpha $ takes all values between $0$
and $1/2$. When $X$ is finite-dimensional, Wickstead showed that the
only values $\alpha $ can possibly take are $\alpha =0$ or $\alpha
=1/n$, where $n$ is a divisor of the dimension of $X$. It was left
open whether $\alpha $ can take any other values or not. In
\cref{cor:wickstead_question} we show that it cannot. The unital
two-dimensional examples announced above are then the cones in
Wickstead's family corresponding to the forbidden values of $\alpha
$.

The result on positive projections will follow from a version of it
for positive projections with constant diagonal on general $C(K)$
spaces (\cref{thm:main}). As a consequence, we will derive a
condition guaranteeing that a Banach lattice algebra does not admit
any faithful representation---unital or not---on a space of regular operators on a
Dedekind complete Banach lattice (\cref{prop:poison}). We will also
generalize Wickstead's complete description of contractive positive
projections with constant diagonal on finite-dimensional Banach
lattices to $\el p(A)$, for an arbitrary set $A$ and $1\le p<\infty $
(\cref{prop:el1}).

\subsection{Lean formalization}

The main results of this paper have been formalized in Lean 4 using
version \verb|v0.1.0| of
\gh{https://github.com/davidmunozlahoz/banlat}, a Lean library for
Banach lattices. This version of the library depends on version
\verb|v4.30.0| of Mathlib
\gh{https://github.com/leanprover-community/mathlib4}.
The formalization is available at
\gh{https://github.com/davidmunozlahoz/wickstead}.
Each formalized definition and theorem is accompanied by a tag
linking to the line in the repository that contains its Lean
statement.

\subsection{Preliminaries and notation}

For a compact Hausdorff space $K$,
we denote the space of continuous functions on $K$ by $C(K)$ and its
dual, which we identify with the space of all finite signed regular
Borel measures on $K$, by $M(K)$. For every $f \in C(K)$ and $\mu  \in
M(K)$, we denote the dual pairing by $\left< \mu ,f \right>$, whilst
the value of $\mu $ at a Borel set $B$ will be denoted by $\mu (B)$. The
point mass measure at $t \in K$ will be denoted by $\delta _t$.
Every $f \in C(K)$ has an associated multiplication operator
$M_f\colon C(K)\to C(K)$ defined by $M_f(g)=fg$, for all $g \in C(K)$.

For the basic theory of vector and Banach lattices,
we refer the reader to the monographs
\cite{luxemburg_zaanen1971,lindenstrauss_tzafriri1979}. For the theory
of positive operators and related topics, see
\cite{abramovich_aliprantis2002,aliprantis_burkinshaw2006}; in
particular, a detailed treatment of positive projections
can be found in \cite[Chapter 5]{abramovich_aliprantis2002}.

Now we want to introduce a couple of definitions that are
convenient for this paper and, in order to do so, we have to recall first
some basic facts of positive and regular operators.
Let $X$ be a vector lattice.
An operator $T\colon X\to X$ is \emph{central} if there exists
$\lambda >0$ such that $\pm T\le \lambda \id_X$, where $\id_X$ denotes
the identity of $X$. The \emph{center of $X$} is the set formed by all
central operators, and it is denoted by $\mathcal{Z}(X)$.
The central operators on $C(K)$ are precisely the
multiplication operators.
When $X$ is Dedekind complete, the space of
regular operators $\Lr X$ becomes a Dedekind complete vector lattice
under the Riesz--Kantorovich formulas. In this case,
$\mathcal{Z}(X)$ is the ideal generated by $\id_X$, and it is in fact a
band projection. The projection onto this band $\mathcal{D}\colon \Lr
X\to \Lr X$ is usually called the \emph{diagonal projection} or
\emph{the projection onto the diagonal}. For this reason, we introduce
the following terminology.

\begin{defn}<\gh[HasConstantDiagonal]{https://github.com/davidmunozlahoz/Wickstead/blob/b614cd74d9850d812e5a98b715c36b74ec0989dc/Wickstead/StatementAndSetup.lean\#L103}>
Let $X$ be a Dedekind complete vector lattice, and let $T\colon
X\to X$ be a regular operator. We say that $T$ has \emph{constant
diagonal $\alpha $} if $\mathcal{D}(T)=\alpha \id_X$.
\end{defn}

Notice that $T$ having constant diagonal $\alpha $ is the same as
$\alpha \id_X$ being the supremum of all the central operators that lie
below $T$. This is the characterization we used at the beginning of
the introduction.

A \emph{lattice-ordered algebra} is a vector lattice together with a
real algebra structure for which the product of positive elements is
positive. The monograph \cite[Chapter
XVII]{birkhoff1967} and the surveys
\cite{huijsmans_luxemburg1995,keimel1995} provide an overview of lattice-ordered
rings and algebras. A \emph{Banach lattice algebra}
\gh[BanachLatticeAlgebra]{https://github.com/davidmunozlahoz/Wickstead/blob/b614cd74d9850d812e5a98b715c36b74ec0989dc/Wickstead/PropositionFourThree.lean\#L44} is a lattice-ordered algebra
together with a complete lattice norm that is submultiplicative. An
introduction to the theory of Banach lattice algebras can be found in
\cite{wickstead2017_questions}.

\begin{defn}<\gh[LatticeAlgebraRepresentation]{https://github.com/davidmunozlahoz/Wickstead/blob/b614cd74d9850d812e5a98b715c36b74ec0989dc/Wickstead/PropositionFourThree.lean\#L106}>
A \emph{representation} of a lattice-ordered algebra $A$ is a
lattice and algebra homomorphism $\Phi \colon A\to \Lr X$, for some
Dedekind complete vector lattice $X$. We say that $\Phi $ is
\emph{faithful} if it is injective. In the case that $A$ has an
algebraic identity $e$ and $\Phi (e)=\id_X$, we say that $\Phi $ is
\emph{unital}.
\end{defn}

\section{The main theorem}

This section is devoted to proving the following.

\begin{thm}<\gh[main_CK]{https://github.com/davidmunozlahoz/Wickstead/blob/b614cd74d9850d812e5a98b715c36b74ec0989dc/Wickstead/FinalAssembly.lean\#L90}>\label{thm:main}
Let $K$ be a non-empty compact Hausdorff space. Let $P\colon C(K)\to C(K)$
be a positive projection with $P\one_K=\one_K$. Suppose
there exists some $\alpha \ge 0$ satisfying:
\begin{enumerate}
    \item $\alpha \id_{C(K)}\le P$, and
    \item if $0\le M_f\le P$, then $\|f\|_\infty \le \alpha $.
\end{enumerate}
Then either $\alpha =0$ or $\alpha =1/n$ for some $n \in \N$.
\end{thm}

We will construct the proof incrementally along this section. We start
with some lemmas that may be of interest by themselves.

\begin{lem}\label{lem:inv_ideal}
Let $X$ be a Banach lattice, and let $P\colon X\to X$ be a positive
projection. Let $S\subseteq X_+$ be a set, and let $I$ be the ideal
generated by $S$. If $P(S)\subseteq I$, then also $P(I)\subseteq
I$ and $P(\bar{I})\subseteq \bar{I}$.
\end{lem}
\begin{proof}
If $x \in I$, then there exist $s_1,\ldots ,s_n \in S$ and
$\lambda _1,\ldots ,\lambda _n \in \R_+$ such that
\[
    |x|\le \lambda _1s_1+ \cdots +\lambda _n s_n.
\]
Using the positivity of $P$ it follows that
\[
|Px|\le P|x|\le \lambda _1Ps_1 + \cdots + \lambda _n Ps_n.
\]
Since $P(S)\subseteq I$, the right hand side in previous equation
belongs to $I$, and therefore $Px \in I$. Thus $P(I)\subseteq I$
and, since $P$ is continuous, $P(\bar{I})\subseteq \bar{I}$.
\end{proof}

\begin{lem}\label{lem:inv_bandcomp}
Let $X$ be a Banach lattice with a quasi-interior point. Let
$P\colon X\to X$ be a positive projection satisfying $\alpha
\id_X\le P$, for some $\alpha >0$. If $B\subseteq X$ is a
projection band invariant under $P$, then so is $B^{d}$.
\end{lem}
\begin{proof}
Let $e$ be a quasi-interior point of $X$, and let $f=Pe$. Since $f=Pe\ge \alpha
e$, and $\alpha >0$, $f$ is also a quasi-interior point. Let
$Q\colon X\to X$ be the band projection onto $B$. Then using the
positivity of $P$, and the fact that $Qf\le f$, it follows that
\[
PQf\le Pf=f.
\]
Since $PQf \in B$, applying $Q$ to this equality yields $PQf\le
Qf$. Thus
\[
0=P(Qf-PQf)\ge \alpha (Qf-PQf)\ge 0
\]
and the equality $PQf=Qf$ follows. Using this, the computation
\[
    (\id_X-Q)f=f-Qf=Pf-PQf=P(\id_X-Q)f
\]
shows that $(\id_X-Q)f$ belongs to the image of $P$.
By \cref{lem:inv_ideal}, the closed ideal generated by $(\id_X-Q)f$ is
invariant under $P$. Since $f$ is a quasi-interior point, this closed ideal is
precisely $B^{d}$. Indeed, for every $x \in X$, find a sequence
$(x_n)$ contained in the ideal generated by $f$ and such that
$x_n\to x$ (this we can do because $f$ is a quasi-interior point).
Then $(\id_X-Q)x_n \to (\id_X-Q)x$, where
the sequence $((\id_X-Q)x_n)$ belongs to the ideal generated by $(\id_X-Q)f$.
\end{proof}

Recall that a projection $P\colon X\to X$ on a vector lattice $X$ is
said to be \emph{strictly positive} if $x>0$ implies $Px>0$.

\begin{lem}\label{lem:rank_one}
Let $X$ be a Banach lattice, and let $P\colon X\to X$ be a
strictly positive projection. If the only closed ideals
invariant under $P$ are
the trivial ones (that is, $\{0\}$ and $X$), then $P$ has rank $1$.
\end{lem}
\begin{proof}
Since $P$ is strictly positive, by \cite[Lemma
5.43]{abramovich_aliprantis2002} $P(X)$ is a (closed) vector sublattice
of $X$. If $P(X)$ didn't have dimension $1$, then there would
exist two positive disjoint vectors $x,y \in P(X)$. The closed ideal
generated by $x$ in $X$ is non-trivial (because $y$ does not
belong to the band generated by $x$), and yet invariant by
\cref{lem:inv_ideal}. This contradicts the assumption.
\end{proof}

Let $K$ be a compact Hausdorff space, and let $P\colon C(K)\to C(K)$
and $\alpha \ge 0$ be as in \cref{thm:main}. If $\alpha =0$ we are
done, so we may assume from now on that $\alpha >0$.
The adjoint $P^{*}\colon M(K)\to M(K)$ is again a positive projection
satisfying $\alpha \id_{M(K)}\le P^{*}$. For every $t \in K$, let $\mu
_t=P^{*}\delta _t$. Notice that the map
\[
\begin{array}{cccc}
& K & \longrightarrow & M(K) \\
    & t & \longmapsto & \mu _t \\
\end{array}
\]
is continuous for the weak$^*$ topology on $M(K)$; this fact will be
used several times later on.

Decompose
\[
\mu _t=\sum_{s \in J_t}^{}\lambda (t,s)\delta _s + \nu _t,
\]
where $\nu _t$ is a positive continuous measure (i.e., it is disjoint from every
point mass measure), $J_t\subseteq K$ is an at most countable set of points, and
$\lambda (t,s)>0$ for every $s \in J_t$. Moreover, since $\alpha
\id_{M(K)}\le P^{*}$ and $\alpha >0$, we must have $t \in J_t$ and $\alpha \le \lambda (t,t)$ for all $t \in K$.
Also, it is important to keep in mind that $\mu _s$ is positive and
\[
\|\mu _s\| = \mu _s(K)=\left< P^{*}\delta _s,\one_K \right> = \left<
\delta _s,P\one_K \right> = \left< \delta _s,\one_K \right> = 1.
\]

\begin{lem}
For every $s \in J_t$, $J_s \subseteq J_t$.
\end{lem}
\begin{proof}
Using that $P^{*}$ is a projection, we have
\begin{equation}\label{eq:image_mu}
\sum_{s \in J_t}^{}\lambda (t,s)\delta _s+\nu _t=\mu _t=P^{*}\mu _t=\sum_{s \in J_t}^{}\lambda (t,s) \mu _s + P^{*}\nu _t.
\end{equation}
Notice that all the coefficients are positive. Hence, for this
equality to hold, the discrete part of each $\mu _s$ cannot
contain any point mass measure supported outside $J_t$.
\end{proof}

Fix an arbitrary $t \in K$. Denote by $I$ the order ideal generated by
$\{\, \mu _s : s \in J_t\, \} $ in $M(K)$. Since $P^{*}\mu _s=\mu _s$
for every $s \in J_t$, \cref{lem:inv_ideal} guarantees that both $I$ and
$\bar{I}$ are invariant under $P^{*}$. Since $M(K)$ is
order continuous, $\bar{I}$ is actually a projection band.

\begin{lem}\label{lem:invariant_ideals_trivial}
The only closed invariant ideals of $P^{*}|_{\bar{I}}$ are the
trivial ones.
\end{lem}
\begin{proof}
We want to use
\cref{lem:inv_bandcomp}. For this, first note that $\bar{I}$, being a
band in $M(K)$, is itself an order continuous Banach lattice
(actually, it is an AL-space). Also note that $\bar{I}$ has a
quasi-interior point: if $J_t=\{t_i\}_{i=1}^{\infty }$ is
an enumeration of $J_t$, then $e=\sum_{i=1}^{\infty}2^{-i}\mu
_{t_i}$ is a quasi-interior point of $\bar{I}$ (recall that $\|\mu
_s\|=1$).

Suppose $I'\subseteq \bar{I}$ is a closed ideal invariant under
$P^{*}|_{\bar{I}}$. The order
continuity of $\bar{I}$ implies that $I'$ is a projection band in $I'$.
By \cref{lem:inv_bandcomp}, its disjoint complement in
$\bar{I}$ is also invariant under $P^{*}|_{\bar{I}}$. Now $\delta _t
\in \bar{I}$ is an
atom, and therefore must belong to either $I'$ or its disjoint
complement. Without loss of generality, assume it belongs to $I'$.
Then the invariance of $I'$ under $P^{*}$ implies that
\[
P^{*}\delta _t=\mu _t=\sum_{s \in J_t}^{}\lambda (t,s)\delta
_s+\nu _t \in I'.
\]
Since $\lambda (t,s)>0$ for every $s \in J_t$,
and $I'$ is an ideal, $\delta _s \in I'$ and, using again
invariance, $\mu _s \in I'$ for all $s \in J_t$. But
$\bar{I}$ is precisely the closed ideal generated by $\mu _s$, so
that $I'=\bar{I}$.
\end{proof}

By \cref{lem:rank_one}, $P^{*}(\bar{I})$ has rank one. That is, there
exist a measure $\rho \in M(K)_+$ and scalars $\lambda _s \in \R_+$
such that $\mu _s=P^{*}\mu _s=\lambda _s \rho $ for all $s \in J_t$.
But since $\mu _s(K)=1$ for every $s \in J_t$, $\lambda _s$ does not
depend on $s$, and therefore $\mu _s=\mu _t$ for every $s \in J_t$.
This implies that the family of all distinct sets $J_t$ forms a
partition of $K$.

\begin{lem}\label{lem:relation_Js}
For all $s \in K$, either $J_s=J_t$ (when $s \in J_t$) or $J_s \cap
J_t=\O$ (when $s \not\in J_t$).
\end{lem}
\begin{proof}
When $s \in J_t$, we have noted already that $\mu _s=\mu _t$, so
obviously $J_s=J_t$. If $s' \in J_s \cap
J_t$, then again we would have $\mu _s=\mu _{s'}=\mu _t$, and
therefore $s \in J_s=J_t$.
\end{proof}

The fact that $\mu _s=\mu _t$ for all $s \in J_t$ has further
implications on the expression of $\mu _t$. First, since
the continuous parts of $\mu _t$ and $\mu _s$ must be equal, $\nu _t=\nu
_s$ holds for every $s \in J_t$. Second, since the coefficient of each
point evaluation must be equal, $\lambda (t,s)=\lambda (s,s)$ holds
for every $s \in J_t$.
Hence $\lambda (t,s)\ge \alpha $ for every $s \in J_t$. Moreover,
since $\alpha >0$ and $\|\mu _s\|=1$, $J_t$ must be finite and satisfy
$|J_t|\le 1/\alpha $.

Next we want to show that $\nu _t=0$. For this we need the following
topological lemma.

\begin{lem}\label{lem:small_pstarnu}
For every $\varepsilon >0$ and $t' \in J_t$, there exists an
open neighborhood $V$ of $t'$ such that $(P^{*}\nu _t)(V)<\varepsilon $.
\end{lem}
\begin{proof}
Let $J_t=\{t_1,\ldots ,t_n\}$, and let $U_1',\ldots ,U_n'$ be
pairwise disjoint open neighborhoods of $t_1,\ldots ,t_n$,
respectively. Let $U'=U_1'\cup \cdots \cup U_n'$.
Since $\nu _t$ is a regular continuous measure, we may
choose these neighborhoods small enough so as to have $\nu
_t(U')<\varepsilon /2$. Let also $U_i$ be an open neighborhood of
$t_i$ satisfying $\overline{U_i}\subseteq U_i'$ for $i=1,\ldots
,n$, and denote $U=U_1\cup \cdots \cup U_n$.

Fix $t' \in J_t$.
For every $s \not\in U$, $J_s$ does not contain any of the points
in $J_t$ by \cref{lem:relation_Js}. Therefore there exists an
open neighborhood $V_s$ of $t'$ such that $\mu
_s(V_s)<\varepsilon /2$. Recall that the map $s'\mapsto \mu _{s'}$
is continuous in the weak$^*$ topology. Hence there exists an open
neighborhood $W_s$ of $s$ such that $\mu _{s'}(V_s)<\varepsilon /2$
for all $s' \in W_s$.

This way the $\{W_s\}_{s \in K\setminus U}$ form an open covering
of the set $K\setminus U$. Since the latter is compact, we can
extract a finite subcovering:
\[
K\setminus U\subseteq W_{s_1}\cup \cdots \cup W_{s_m}.
\]
Consider the set $V=V_{s_1}\cap \cdots \cap V_{s_m}$, which is
again an open neighborhood of $t'$. Notice that, if $s \not\in U$, then
$s \in W_{s_i}$ for some $i \in \{1,\ldots ,m\}$, and therefore
\[
    \mu _s(V)\le \mu _s(V_{s_i})<\varepsilon /2.
\]

To prove that $V$ is the set we wanted to construct,
we need to show that, for every $f \in C(K)$ satisfying
$0\le f\le \one_K$ and $\supp f \subseteq V$, we have
\[
\left< P^{*}\nu _t,f \right> =\left< \nu _t, Pf \right> < \varepsilon .
\]
Fix such an $f \in C(K)$. Let $U_1=U'$, $U_2=K\setminus \bar{U}$,
and let $\{f_1,f_2\}$ be a partition of unit subordinate to the
open covering $\{U_1,U_2\}$. On the one hand, for every $s \not\in U$:
\[
    (Pf)(s)f_2(s)\le \left< P^{*}\delta _s,f \right> = \left< \mu
    _s, f\right> \le \mu _s(V)<\varepsilon /2.
\]
Since $f_2(s)=0$ for $s \in U$, this shows that
$(Pf)f_2<\varepsilon /2 \one_K$. In particular,
\[
    \left< \nu
    _t,(Pf)f_2 \right> \le \varepsilon /2 \nu _t(K)\le \varepsilon /2.
\]
On the other hand, since
$\|Pf\|_\infty \le \|f\|_\infty \|P\one_K\|_\infty \le 1$,
\[
\left< \nu _t, (Pf)f_1 \right> \le \left< \nu _t, f_1 \right> \le
\nu _t(U') <\varepsilon /2.
\]
Therefore
\[
\left< \nu _t, Pf \right> = \left< \nu _t, (Pf)f_1 \right> +
\left< \nu _t, (Pf)f_2 \right> < \varepsilon.\qedhere
\]
\end{proof}

\begin{lem}
For every $t \in K$, $\nu _t=0$.
\end{lem}
\begin{proof}
Expanding the right hand side of \eqref{eq:image_mu}, we have the following identity:
\begin{equation}\label{eq:image_mu_expanded}
\sum_{s \in J_t}^{}\lambda (t,s)\delta _s+\nu _t=\sum_{s \in
J_t}^{}\sum_{s' \in J_t}^{} \lambda (t,s)\lambda (s,s')\delta
_{s'}+ \sum_{s \in J_t}^{}\lambda (t,s)\nu _t + P^{*}\nu _t,
\end{equation}
where we are using that $J_s=J_t$ and also that $\nu _s=\nu _t$
for every $s \in J_t$. Fix
$t' \in J_t$. According to \cref{lem:small_pstarnu}, for every
$\varepsilon >0$ there exists an open
neighborhood $V$ of $t'$ such that $(P^{*}\nu _t)(V)<\varepsilon $.
By making $V$ smaller, we can assume that also $\nu_t
(V)<\varepsilon $ and that $t'' \not\in V$ for every other $t'' \in
J_t$ different from $t'$. Evaluating the above expression at $V$
yields
\[
\lambda (t,t')+\nu _t(V)=\sum_{s \in J_t}^{}\lambda (t,s)\lambda
(s,t')+\sum_{s \in J_t}^{}\lambda (t,s)\nu _t(V)+(P^{*}\nu _t)(V).
\]
Using that $\sum_{s \in J_t}^{}\lambda (t,s)\le 1$, rearrange the
previous expression to get
\[
\bigg| \lambda (t,t')-\sum_{s \in J_t}^{}\lambda (t,s)\lambda
(s,t') \bigg| < 3\varepsilon.
\]
Since $\varepsilon >0$ was arbitrary, it follows that
$\lambda (t,t')=\sum_{s \in J_t}^{}\lambda (t,s)\lambda (s,t')$.

Using these identities for every $t' \in J_t$,
\eqref{eq:image_mu_expanded} becomes
\[
\nu _t=\sum_{s \in J_t}^{} \lambda (t,s)\nu _t + P^{*}\nu _t.
\]
Applying $P^{*}$ to the above equation, and the fact that $\sum_{s \in
J_t}^{} \lambda (t,s)>0$, yields $P^{*}\nu _t=0$. Therefore
\[
\nu _t=\sum_{s \in J_t}^{} \lambda (t,s)\nu _t
\]
and either
$\nu _t=0$, in which case we are done, or $\sum_{s \in J_t}^{} \lambda
(t,s)=1$. But since
\[
1=\mu _t(K)=\sum_{s \in J_t}^{} \lambda (t,s) + \nu _t(K),
\]
in the latter case we also end up having $\nu _t=0$.
\end{proof}

So far we have proved that, for every $t \in K$,
\[
\mu _t=\sum_{s \in J_t}^{}\lambda (t,s)\delta _s,
\]
where $\alpha \le \lambda (t,s)$ and $J_t$ is finite with uniformly
bounded size $|J_t|\le 1/\alpha $. Up to this point we have not used
that $\alpha \id_{C(K)}$ is the biggest multiplication operator below
$P$, only that $\alpha \id_{C(K)}\le P$. Now this fact has
to come into play.

\begin{lem}\label{lem:dense}
For every $n \in \N$, there exists a dense set $D_n\subseteq K$ such
that $\lambda (t,t)\le \alpha +1/n$ for all $t \in D_n$.
\end{lem}
\begin{proof}
Let $n \in \N$. In order to get a contradiction, suppose there
exists an open set $U\subseteq K$ such that $\lambda (t,t)>\alpha
+1/n$ for all $t \in U$. Let $V$ be an open set with
$\bar{V}\subseteq U$. By Urysohn's lemma, there exists a continuous
function $f \in C(K)$ such that $0\le f\le \one_K$,
$\supp f\subseteq U$, and $f|_V=1$.

Set $g=(\alpha +1/n)f$. Then for every $h \in C(K)_+$ and $t \in
U$:
\[
g(t)h(t)\le (\alpha +1/n)h(t)\le \lambda (t,t)h(t)\le (Ph)(t).
\]
Also, $g(t)h(t)=0\le (Ph)(t)$ whenever $t
\not\in U$. Hence $M_g\le P$, but $\|g\|_\infty =\alpha
+1/n>\alpha $, contradicting the assumption.
\end{proof}

We can now use this fact to obtain the final lemma.

\begin{lem}
There exists $t \in K$ such that $\lambda (t,s)=\alpha $ for all
$s \in J_t$.
\end{lem}
\begin{proof}
Since $|J_s|\le 1/\alpha $ for all $s \in K$, we can choose $t \in
K$ for which the cardinality of $J_t$ is maximal. We claim that
$\lambda (t,s)=\alpha $ for all $s \in J_t$. To show this, we will
proceed by contradiction: suppose there exist some $t_1\in J_t$
and some $n \in \N$ for which $\lambda (t,t_1)>\alpha + 1/n$.

Since $J_{t_1}=J_t$, $J_{t_1}$ also has maximal cardinality.
Write $J_t=\{t_1,\ldots ,t_m\}$, and let $U_1,\ldots ,U_m$ be
pairwise disjoint open neighborhoods of $t_1,\ldots ,t_m$.
By the weak$^*$ continuity of $s \mapsto \mu _s$, there exists an
open neighborhood $U$ of $t_1$ such that, whenever $s \in U$,
$\mu _s(U_1)>\alpha +1/n$ (recall that $\mu _{t_1}(U_1)>\alpha
+1/n$) and
$\mu _s(U_i)>\alpha /2$ (recall that $\mu _{t_1}(U_i)\ge \alpha $) for $i=2,\ldots ,m$. This means that, for
every $s \in U$, $J_s$ must contain at least a point in each $U_1,\ldots
,U_m$. Since these sets are pairwise disjoint, $J_s$ contains at
least $m$ points. The maximality of $m$ implies then that $J_s$
must contain exactly $m$ points, one in each of the sets
$U_1,\ldots ,U_m$.

In particular, for the $s \in U\cap U_1$, the point of $J_s$ that
belongs to $U_1$ is precisely $s$. And since $\mu _s(U_1)>\alpha
+1/n$, and the only point of $J_s$ that belongs to $U_1$ is $s$, it
follows that $\lambda (s,s)>\alpha +1/n$ for every $s \in U\cap U_1$.
This contradicts \cref{lem:dense}.
\end{proof}

Pick $t \in K$ as in the lemma above. Then
\[
\mu _t=\sum_{s \in J_t}^{}\alpha \delta _s
\]
so that
\[
1=\mu _t(K)=\sum_{s \in J_t}^{} \alpha =\alpha |J_t|
\]
which implies $\alpha =1/|J_t|$. This completes the proof of
\cref{thm:main}.

We conclude this section with a general example showing how each of
the values $0$ and $1/n$, for $n \in \N$, can be attained.

\begin{example}
Let $K$ be a compact Hausdorff space. Let $G$ be a finite subgroup of the
homeomorphism group of $K$. Then the map $P\colon C(K)\to C(K)$
defined by
\[
Pf=\frac{1}{|G|}\sum_{\sigma \in G}^{} f\circ
\sigma,\quad\text{for }f \in C(K),
\]
is a positive projection that satisfies the conditions of
\cref{thm:main} with $\alpha =1/|G|$. Indeed,
it is immediate to check that $P$ is a positive projection such
that $1/|G| \id_{C(K)}\le P$ and $P\one_K=\one_K$. Suppose
$0\le M_h\le P$ for some $h \in C(K)_+$. Given $t \in K$,
choose an open neighborhood $U$ of $t$ such that $\sigma (t)
\not\in U$ for all $\sigma \in G\setminus \{\id_K\}$, and let $f
\in C(K)$ be a continuous function satisfying $f(t)=1$, $0\le f\le
\one_K$, and $\supp f \subseteq U$. Then
\[
h(t)=M_h(f)(t)\le (Pf)(t)=\frac{1}{|G|}.
\]
Since this is true for every $t \in K$, $\|h\|_\infty \le 1/|G|$.

To construct an example with $\alpha =0$, let $\mu \in
M(K)_+$ be a continuous measure with $\mu (K)=1$, and define
\[
Pf=\one_K\int_{K}^{} f\, d \mu\quad\text{for }f \in C(K).
\]
It is immediate to check that $P$ is a positive projection
satisfying $P\one_K=\one_K$.
Suppose $0\le M_h\le P$, for some $h \in C(K)_+$. Fix $t \in K$
and $\varepsilon >0$. Let $U$ be an open neighborhood of $t$ for which
$\mu (U)\le \varepsilon $. Let $f \in C(K)$ be a continuous
function satisfying $f(t)=1$, $0\le f\le
\one_K$, and $\supp f \subseteq U$. Then
\[
h(t)=M_h(f)(t)\le (Pf)(t)\le \mu (U)<\varepsilon.
\]
Since $\varepsilon >0$ and $t \in K$ were arbitrary, $h=0$.
\end{example}

\section{Wickstead's conjecture and other consequences}

\Cref{thm:main} can now be used to answer \cite[Question
7.3]{wickstead2017_two}.

\begin{cor}<\gh[main]{https://github.com/davidmunozlahoz/Wickstead/blob/b614cd74d9850d812e5a98b715c36b74ec0989dc/Wickstead/CorollaryThreeOne.lean\#L1149}>\label{cor:wickstead_question}
Let $X$ be a non-zero Dedekind complete vector lattice. Let $P\colon X\to
X$ be a positive projection with constant diagonal
$\alpha $. Then either $\alpha =0$ or $\alpha =1/n$, for some $n
\in \N$.
\end{cor}
\begin{proof}
If $\alpha =0$ we are done, so assume $\alpha >0$. In particular,
$P\neq 0$, so there exists some $e \in X_+$ such that $Pe\neq 0$.
Then $I_{Pe}$, the ideal generated by $Pe$, is invariant under $P$
by \cref{lem:inv_ideal}.
Use Kakutani's theorem \cite[Theorem 1]{kakutani1941} to identify $(I_{Pe},\|{\cdot }\|_{Pe})$
lattice isometrically with $C(K)$,
for some non-empty compact Hausdorff space $K$, in such a way that
$Pe$ is identified with the function $\one_K$. We are going to
show that the restriction of $P$ to $I_{Pe}$ satisfies the
conditions of \cref{thm:main}.

It is clear that the restriction is a positive projection
satisfying $\alpha \id_{C(K)}\le P$. Moreover,
$P\one_K=P(Pe)=Pe=\one_K$. Since $X$ is Dedekind complete, and the
diagonal of $P$ is $\alpha \id_{X}$, we can express $P=\alpha
\id_X+T$, where $T$ is a positive operator disjoint from $\id_X$. By
the Riesz--Kantorovich formulas \cite[Theorem
1.18]{aliprantis_burkinshaw2006}, this means that
\[
    0=(T\wedge \id_X)(x)=\inf_{0\le y\le x} Ty+(x-y)
\]
holds for every $x \in X_+$. Note that $T=P-\alpha \id_X$ also
leaves $I_{Pe}$ invariant. Moreover, since $I_{Pe}$ is itself a
Dedekind complete Banach lattice, the Riesz--Kantorovich formulas
also hold in $\Lr{I_{Pe}}$. Therefore, for every $x \in
(I_{Pe})_+$,
\[
    (T|_{I_{Pe}}\wedge \id_X)(x)=\inf_{y \in I_{Pe}, 0\le y\le x}
    Ty+(x-y)=\inf_{0\le y\le x} Ty+(x-y)=0,
\]
where in the second equality we are using that $I_{Pe}$ is an
ideal.
Hence, in $\Lr{I_{Pe}}$, $P|_{I_{Pe}}=\alpha \id_{I_{Pe}}+T|_{I_{Pe}}$, with $T$
disjoint from the identity. Now, if $f \in C(K)_+$ is such that
$M_f\le P$, then because of the previous decomposition and the
fact that $M_f$ is a central operator $M_f\le
\alpha \id_{C(K)}$. Hence $\|f\|_\infty \le \alpha $.
\end{proof}

Previous result applies, in particular, to conditional
expectations with constant diagonal. We can use this fact, together
with the ideas in the proof of \cref{thm:main}, to
characterize all contractive positive projections with constant
diagonal on $\el p$ spaces. Given a set $A$, $\chi _B$ will denote
the characteristic function of $B\subseteq A$, and $e_a$ will denote
the characteristic function of the singleton $\{a\}\subseteq A$.

\begin{prop}\label{prop:el1}
Let $A $ be a non-empty set, let $1\le p<\infty $, and let $P\colon \el p(A)\to \el
p(A)$ be a contractive positive projection with constant diagonal
$\alpha >0$. Then $n=1/\alpha \in \N$, and there exists a partition
$\{A_i\}_i$ of $A $ into sets of size $n$ such that
\[
P\bigg(\sum_{a \in A}^{}\lambda _a e_a\bigg) = \alpha \sum_{i}^{} \bigg(\sum_{a \in A_i}^{} \lambda _a\bigg) \chi
_{A_i},
\]
whenever $\sum_{a \in A}^{}|\lambda _a |^{p}<\infty $.
In particular, if $A$ is countable, then there exists an enumeration of
$A$ such that the matrix associated with $P$ is
\[
    \begin{pmatrix}
        \alpha \bm{1}_n& \bm{0}_n&\bm{0}_n&\cdots\\
        \bm{0}_n& \alpha \bm{1}_n&\bm{0}_n&\cdots\\
        \bm{0}_n&\bm{0}_n&\alpha \bm{1}_n&\cdots\\
        \vdots &\vdots & \vdots&\ddots
    \end{pmatrix},
\]
where $\bm{0}_n$ denotes an $n$ by $n$ block of zeros, and
$\bm{1}_n$ denotes an $n$ by $n$ block of ones.
\end{prop}
\begin{proof}
For every $a \in A$, we can express
\[
Pe_a=\sum_{b \in J_a}^{} \lambda (a,b) e_b,
\]
where $\lambda (a,b)>0$ and $J_a \subseteq A$ is at most
countable. Moreover, that the diagonal of $P$ is $\alpha $ implies
$(Pe_a)(a)=\alpha $, that is, $a
\in J_a$ and $\lambda (a,a)=\alpha $.

Fix $a \in A$. We claim that, for every $b \in J_a$, $J_b
\subseteq J_a$. Indeed, since $P$ is a projection,
\[
\sum_{b \in J_a}^{}\lambda (a,b)e_b=Pe_a=P(Pe_a)=\sum_{b \in
J_a}^{}\lambda (a,b)\sum_{c \in J_b}^{}\lambda (b,c)e_c.
\]
Since all the coefficients are positive,
it cannot be the case that $J_b$, for $b \in J_a$, introduces a
new coordinate. Hence $J_b \subseteq J_a$.

Let $I$ be the closed linear span of $\{\, e_b : b \in J_a \, \}$. Notice
that $I$ is a closed ideal, and therefore a band. By the above
observation and \cref{lem:inv_ideal}, this ideal is invariant
under $P$. We are going to show that the restriction $P|_I\colon
I\to I$ has rank $1$ just like in
\cref{lem:invariant_ideals_trivial}. By \cref{lem:rank_one}, it suffices to show
that the only invariant closed ideals are the trivial ones.
Let $I'\subseteq I$
be a closed ideal invariant under $P|_I$. Since $I$ is
order continuous and has a quasi-interior point (in fact, it is
lattice isometric to $\el p$), it follows from
\cref{lem:inv_bandcomp} that the
disjoint complement of $I'$ in $I$ is also invariant. But $e_a$ is
an atom: it must belong to either $I'$ or its disjoint complement.
Suppose, without loss of generality, that $e_a \in I'$. Then by
invariance $Pe_a \in I'$ and, being $I'$ an ideal, $e_b \in J_a$
for all $b \in J_a$. Since $I$ is the smallest closed ideal
containing all the $e_b \in J_a$, it follows that $I'=I$.
Hence $P|_I$ has rank one.

Let $x$ be a positive norm-one element in $P(I)$. Notice that, in
particular, $x(b)>0$ for all $b \in J_a$. Define a probability measure $\mu
$ on $J_a$ by $\mu (b)=x(b)^{p}$, for $b \in J_a$. Then $I$ is lattice
isometric to $L_p(\mu )$ through the map
\[
\begin{array}{cccc}
J\colon& I & \longrightarrow & L_p(\mu ) \\
        & y & \longmapsto & (y(b)x(b)^{-1})_{b \in J_a} \\
\end{array}.
\]
Define $\tilde P=JP|_IJ^{-1}\colon L_p(\mu )\to L_p(\mu )$. This is
still a contractive positive projection, and it satisfies
\[
\tilde P(\one_{J_a})=JPJ^{-1}(\one_{J_a})=JPx=Jx=\one_{J_a}.
\]
By the Andô--Douglas theorem \cite[Corollary
5.53]{abramovich_aliprantis2002}, $\tilde P$ is the conditional expectation with respect to
some $\sigma $-algebra $\Sigma $ on $J_a$. Since $J_a$ is at most
countable, $\Sigma $ must be generated by a partition $\{A_i\}$ of
$J_a$. Hence the conditional expectation has the form
\[
\tilde P(f)=\sum_{i}^{} \chi _{A_i} \frac{1}{\mu (A_i)}\sum_{b \in
A_i}^{}\mu (b)f(b)\quad\text{for all }f \in L_p(\mu ).
\]
From $\tilde P(L_p(\mu ))=\barespn\{\one_{J_a}\}$ it
follows that the partition must be trivial, so that actually
\[
\tilde P(f)=\one_{J_a} \sum_{b \in J_a}^{} \mu
(b)f(b)\quad\text{for all }f \in L_p(\mu ),
\]
and therefore
\[
\begin{aligned}
P|_I(y)
  &=J^{-1}\tilde P Jy\\
  &=J^{-1}(\one_{J_a})\sum_{b \in J_a}^{}\mu(b)x(b)^{-1}y(b)\\
  &=x \sum_{b \in J_a}^{}x(b)^{p-1}y(b)
    \quad\text{for all }y \in I.
\end{aligned}
\]
In particular, $\alpha =(Pe_b)(b)=x(b)^{p}$ for all $b \in J_a$. Since
$\alpha >0$, and $\|x\|_p=1$, this means that $J_a$ must be finite
with $|J_a|=1/\alpha $. Also, the expression for $P|_I$ ends up
being
\begin{equation}\label{eq:Pel1}
P|_I(y)=\chi _{J_a}\alpha \sum_{b \in J_a}^{}y(b)\quad\text{for all }y \in I.
\end{equation}
Therefore, for every $b \in J_a$, $Pe_b=\alpha \chi
_{J_a}=Pe_a$. Hence $J_a=J_b$ whenever $b \in J_a$, and $J_a\cap
J_b=\O$ otherwise.

Consider then the partition $\{A_i\}$ formed by the sets of the
form $J_a$, $a \in A$, that are different. Using
\eqref{eq:Pel1} we arrive at
\[
P\bigg( \sum_{a \in A}^{}\lambda _a a \bigg) =P \bigg(
\sum_{i}^{}\sum_{a \in A_i}^{}\lambda _a a \bigg) =\alpha
\sum_{i}^{} \bigg( \sum_{a \in A_i}^{}\lambda _a \bigg) \chi
_{A_i}
\]
whenever $\sum_{a \in A}^{}|\lambda _a|^{p}<\infty $. In
the case that $A$ is at most countable, to get the desired matrix
first enumerate the partition $\{A_{k}\}_{k \in \N}$, and then
enumerate $A$ in such a way that $A_k=\{n(k-1)+1,\ldots ,nk\}$,
where $n=1/\alpha $.
\end{proof}

\begin{rem}
Following the ideas of previous proof, it is not difficult to check
that the only contractive positive projection $P\colon \el p(A)\to
\el p(A) $ with constant diagonal $0$ is $P=0$.
\end{rem}

\section{On the representation of Banach lattice algebras}

In this section we use \cref{cor:wickstead_question} to settle the
representation problem for Banach lattice algebras in the negative.
We begin with two preparatory lemmas: the first records that the
range of a positive projection is itself a Dedekind complete vector
lattice, and the second shows that disjointness from a positive
idempotent in $\Lr X$ becomes disjointness from the identity
on its range. These lemmas will let us apply
\cref{cor:wickstead_question} on the range of an arbitrary positive
idempotent of $\Lr X$.

In the following lemmas, when $X$ is a vector lattice, $Y$ is a
sublattice, and $S\subseteq Y$ is a set, we denote by $\sup_X S$ the
supremum of $S$ in $X$ (whenever it exists) and by $\sup_Y S$ the
supremum of $S$ in $Y$ (whenever it exists).

\begin{lem}\label{lem:range_lattice}
Let $X$ be a Dedekind complete vector lattice, and let $E\colon
X\to X$ be a positive projection. Then $Y=E(X)$, endowed with the
order inherited from $X$, is a Dedekind complete vector lattice
with supremum $\vee _Y$ given by
\[
y_1\vee_Y y_2=E(y_1 \vee y_2)\quad\text{for all }y_1,y_2 \in Y.
\]
\end{lem}
\begin{proof}
Let $y_1,y_2 \in Y$. Since $E$ is positive and $Ey_i=y_i$ for
$i=1,2$, we have $E(y_1\vee y_2)\ge Ey_i=y_i$, so $E(y_1\vee
y_2)$ is an upper bound of $\{y_1,y_2\}$ in $Y$. If $z \in Y$
satisfies $z\ge y_1$ and $z\ge y_2$, then $z\ge y_1\vee y_2$ in
$X$, whence $z=Ez\ge E(y_1\vee_X y_2)$. This shows that
$y_1\vee_Y y_2=E(y_1\vee y_2)$.

For Dedekind completeness, let $(y_\gamma )$ be an increasing net
in $Y$ that is bounded above by some $y_0 \in Y$. By Dedekind
completeness of $X$, $z=\sup_X y_\gamma $ exists. Then $Ez\ge
Ey_\gamma =y_\gamma $ for all $\gamma $, so $Ez$ is an upper bound
of $(y_\gamma )$ in $Y$. If $w \in Y$ also satisfies $w\ge
y_\gamma $ for all $\gamma $, then $w\ge z$ in $X$, and so $w=Ew\ge
Ez$. Hence $Ez=\sup_Y y_\gamma $.
\end{proof}

\begin{lem}\label{lem:disjointness_transfer}
Let $X$ be a Dedekind complete vector lattice, and let $E,T \in
\Lr X$ be positive operators satisfying:
\begin{enumerate}
    \item $E^2=E$,
    \item $ET=TE=T$, and
    \item $E\wedge T=0$ in $\Lr X$.
\end{enumerate}
Set $Y=E(X)$, endowed with the order inherited from $X$. Then
\[
\id_Y\wedge T|_Y=0\quad\text{in }\Lr Y.
\]
\end{lem}
\begin{proof}
By \cref{lem:range_lattice}, $Y$ is a Dedekind complete vector
lattice, so $\Lr Y$ is itself a Dedekind complete vector lattice
and the Riesz--Kantorovich formulas hold in it. Note also that
$T=ET=TE$ implies that $T$ leaves $Y$ invariant, so the restriction
$T|_Y\colon Y\to Y$ is a well-defined element of $\Lr Y$.

Fix $y \in Y_+$. We compute $(E\wedge T)(y)$ using the
Riesz--Kantorovich formula in $\Lr X$:
\[
(E\wedge T)(y)=\inf_X \{\, Ea+T(y-a) : 0\le a\le y \, \}.
\]
Since $T=TE$, for any $a \in X_+$ with $a\le y$,
\[
Ea+T(y-a)=Ea+Ty-TEa=(\id_Y-T|_Y)(Ea)+T|_Y(y).
\]
As $a$ ranges over $\{a \in X_+ : a\le y\}$, the element $w=Ea$
ranges over the order interval $[0,y]_Y=\{w \in Y_+ : w\le y\}$ in
$Y$. Therefore
\begin{equation}\label{eq:disjointness_X}
(E\wedge T)(y)=\inf_X \{\, (\id_Y-T|_Y)(w)+T|_Y(y) : w \in [0,y]_Y \, \}.
\end{equation}
On the other hand, the Riesz--Kantorovich formula in $\Lr Y$ gives
\begin{equation}\label{eq:RK_Y}
(\id_Y\wedge T|_Y)(y)=\inf_Y\big\{\, (\id_Y-T|_Y)(w)+T|_Y(y) :
w \in [0,y]_Y\,\big\}.
\end{equation}
The set whose infimum is taken is the same subset of $Y$ in both
\eqref{eq:disjointness_X} and \eqref{eq:RK_Y}. The only difference is
that the first infimum is taken in $X$ while the second is taken in
$Y\subseteq X$. Therefore
\[
    (\id_Y \wedge T|_Y)(y)\le (E\wedge T)(y)=0.
\]
Since this holds for every $y \in Y_+$, we conclude that $\id_Y\wedge
T|_Y=0$ in $\Lr Y$.
\end{proof}

We can now state the criterion for non-representability. The
proposition below generalizes the criterion implicit in
\cite[Section 7]{wickstead2017_two} in two ways: the algebra $A$
need not have an algebraic identity, and the conclusion does not
require the representation to be unital.

\begin{prop}<\gh[non_representable]{https://github.com/davidmunozlahoz/Wickstead/blob/main/Wickstead/PropositionFourThree.lean\#L1044}>\label{prop:poison}
Let $A$ be a lattice-ordered algebra. Suppose there exist non-zero
positive elements $e,p \in A_+$ satisfying:
\begin{enumerate}
    \item $e^2=e$, $p^2=p$,
    \item $ep=pe=p$, and
    \item $p=\alpha e+x$, for some $x \in A$ with $x\wedge e=0$, and
        some $\alpha \in \R_+$.
\end{enumerate}
If $\alpha \not\in \{0\}\cup \{1/n:n \in \N\}$, then $A$ has no
faithful representation on $\Lr X$, for $X$ any Dedekind complete vector
lattice.
\end{prop}
\begin{proof}
Suppose that $\Phi \colon A\to \Lr X$ is
a faithful representation of $A$, where $X$ is a Dedekind complete
vector lattice. Set $E=\Phi (e)$, $P=\Phi (p)$, and $T=\Phi (x)$.
Since $\Phi $ is a lattice and algebra homomorphism, $E$ and $P$
are positive projections, and they are non-zero because $\Phi $
is faithful and $e,p\neq 0$. The relations
\[
ex=xe=x,
\]
which follow from $ep=pe=p$ together with $p=\alpha e+x$, give
\[
ET=TE=T.
\]
Applying $\Phi $ to (iii) yields $P=\alpha E+T$, and the
disjointness $e\wedge x=0$ in $A$ transfers via $\Phi $ to
$E\wedge T=0$ in $\Lr X$.

By \cref{lem:disjointness_transfer}, setting $Y=E(X)$, we obtain
$\id_Y\wedge T|_Y=0$ in $\Lr Y$. Restricting the equality $P=\alpha
E+T$ to $Y$ yields
\[
P|_Y=\alpha \id_Y+T|_Y,
\]
where, by \cref{lem:range_lattice}, $Y$ is a Dedekind complete
vector lattice (non-zero, since $E\neq 0$), $P|_Y$ is a positive
projection on it, and $T|_Y$ is disjoint from $\id_Y$. That is,
$P|_Y$ has constant diagonal $\alpha $. By
\cref{cor:wickstead_question}, $\alpha \in \{0\}\cup \{1/n:n \in
\N\}$.
\end{proof}

\begin{rem}
When $e$ is an algebraic identity and $\Phi $ is unital, \cref{lem:disjointness_transfer} is not actually needed:
$\Phi $ being unital then forces $E=\id_X$, so $Y=X$, and the
operator $P=\Phi (p)$ already has constant diagonal $\alpha $ in
$\Lr X$, allowing one to apply \cref{cor:wickstead_question}
directly. The role of \cref{lem:disjointness_transfer} is precisely
to handle the case in which $E=\Phi (e)$ is not the identity of
$\Lr X$, which is what one encounters either when the algebra has
no identity or when the representation is not required to be
unital.
\end{rem}

The following example, due to Wickstead \cite[Section
7]{wickstead2017_two}, exhibits Banach lattice algebras of
dimension $2$ satisfying the conditions of \cref{prop:poison}. In
its original form it answered \cite[Question
5.3]{wickstead2017_questions} and \cite[Problem
2]{wickstead2017_open} in the case of unital representations; with
the strengthened proposition above, the conclusion no longer
depends on the representation being unital.

\begin{example}\label{ex:final}
Let $-1\le \beta \le 0$. Consider $\R^2$ with pointwise multiplication and positive cone
\[
P_{\beta }^{1}=\{\, (x,y)\in \R^2 : 0\le x,\beta x\le y\le x \, \}.
\]
It is shown in \cite[Theorem 3.1]{wickstead2017_two} that this is a lattice cone for
which the product of positive elements is positive. Moreover,
$(\R^2,P_{\beta }^{1})$ can be endowed with a norm so as to make it
a Banach lattice algebra \cite[Proposition 8.3]{wickstead2017_two}. Also, the
algebraic identity $e=(1,1)$ is positive.

Let $p=(1,0)$ and
$x=(1,\beta )$. Note that $p^2=p$, $p\ge 0$, and that
\[
p=\frac{\beta }{\beta -1}e + \frac{1}{1-\beta}x.
\]

Let us check that $x\wedge e=0$. If $y=(y_1,y_2)$ and $x,e\ge y$,
then
\[
\beta (1-y_1)\le \beta -y_2\le 1-y_1\text{ and }\beta (1-y_1)\le
1-y_2\le 1-y_1,
\]
which give $\beta y_1\ge y_2\ge y_1$. Thus $y_1\le 0$, which
together with $\beta y_1\ge y_2\ge y_1$, implies that $y\le 0$.

By \cref{prop:poison}, applied with $e=(1,1)$ and $p=(1,0)$, if
\[
    \frac{\beta }{\beta -1}\not\in \{0\}\cup \{1/n:n \in \N\},
\]
that is, if $\beta \not\in \{0\}\cup \{-1/(n-1):n \in \N,n\ge 2\}$,
then $(\R^2,P_{\beta }^{1})$ admits no faithful representation on
any Dedekind complete vector lattice.
\end{example}

\section*{AI disclosure}

Claude Opus 4.6 provided the details of
\cref{lem:disjointness_transfer} and was also used to correct typos and
minor errors in the final version of the manuscript. All other results
and the exposition are due to the author.

\section*{Acknowledgements}

I would like to thank all the participants in the \textit{Workshop on Free
Banach Lattices} held at the University of Wuppertal in
February 2026. One of the problem sessions at the workshop was devoted
to this question. The enriching discussions provided valuable ideas
and motivated me to pursue the problem further. I am
particularly grateful to J.\ Glück for pointing out
\cref{lem:inv_bandcomp} and for organizing the workshop, and
to P.\ Tradacete for his feedback on early versions of the paper.

\emergencystretch=2em
\printbibliography

\end{document}